\documentclass[reqno]{amsart}
\usepackage{amsfonts}
\usepackage{amsmath, amsthm, amssymb}

\usepackage{comment}
\usepackage{tensor}

\numberwithin{equation}{section}

\newtheorem{Theorem}{Theorem}[section]
\newtheorem{Lemma}[Theorem]{Lemma}

\newcommand{\cP}{\mathcal{P}}
\newcommand{\R}{\mathbb{R}}
\newcommand{\cS}{\mathcal{S}}
\renewcommand{\S}{\mathbb{S}}
\newcommand{\del}{\partial}
\renewcommand{\phi}{\varphi}
\newcommand{\grad}{\nabla}
\newcommand{\Hess}{\operatorname{Hess}}
\newcommand{\Ric}{\operatorname{Ric}}
\newcommand{\tr}{\operatorname{tr}}
\DeclareMathOperator*{\esup}{ess \, sup}
\DeclareMathOperator*{\einf}{ess \, inf}
\renewcommand{\[}{\begin{equation}}
\renewcommand{\]}{\end{equation}}

\title{Pucci eigenvalues on geodesic balls}

\author{Sinan Ariturk}
\address{Pontif\'icia Universidade Cat\'olica do Rio de Janeiro}
\email{ariturk@mat.puc-rio.br}

\begin{document}

\begin{abstract}
We study the eigenvalue problem for the Riemannian Pucci operator on geodesic balls.
We establish upper and lower bounds for the principal Pucci eigenvalues depending on the curvature, extending Cheng's eigenvalue comparison theorem for the Laplace-Beltrami operator.
For manifolds with bounded sectional curvature, we prove Cheng's bounds hold for Pucci eigenvalues on geodesic balls of radius less than the injectivity radius.
For manifolds with Ricci curvature bounded below, we prove Cheng's upper bound holds for Pucci eigenvalues on certain small geodesic balls.
We also prove that the principal Pucci eigenvalues of an $O(n)$-invariant hypersurface immersed in $\R^{n+1}$ with one smooth boundary component are smaller than the eigenvalues of an $n$-dimensional Euclidean ball with the same boundary.
\end{abstract}



\maketitle

\section{Introduction}
We establish two geometric comparison inequalities for the principal Dirichlet half-eigenvalues of the Riemannian Pucci operator.
The Pucci operator is a fully nonlinear elliptic operator with no variational structure.
Let $B$ be an open ball in $\R^n$ of finite radius and let $g$ be a Riemannian metric on $\overline B$.
Let $f: \overline B \to \R$ be a function and let $\mu_1, \mu_2, \ldots, \mu_n$ be the eigenvalues of the Riemannian Hessian of $f$, as defined in \eqref{hessdef} and \eqref{hessevdef} below.
For positive constants $a \le A$, the Riemannian Pucci operator $\cP_g^+$ is defined by
\[
	\cP_g^+ f = \bigg( \sum_{\mu_j > 0} A \mu_j \bigg) + \bigg( \sum_{\mu_j < 0} a \mu_j \bigg)
\]
We consider the following eigenvalue problem with Dirichlet boundary conditions
\[
\label{eveq}
	\begin{cases}
		\cP_g^+ \phi = -\lambda \phi & \text{ over } B \\
		\phi = 0 & \text{ over } \del B \\
	\end{cases}
\]

If $a=A=1$, then $\cP_g^+$ is the Laplace-Beltrami operator.
In this case, it is well known that the eigenvalues of \eqref{eveq} form a sequence
\[
	0 < \lambda_1 < \lambda_2 \le \lambda_3 \le \ldots
\]
The smallest eigenvalue $\lambda_1$ is known as the principal eigenvalue, and it is the unique number such that \eqref{eveq} admits a solution which does not vanish in $B$.
The principal eigenvalue is characterized variationally by
\[
\label{rayleigh}
	\lambda_1 = \inf \bigg\{ \frac{\int_B | \grad f |^2 \,dV}{\int_B |f|^2 \,dV} : f \in C_0^\infty(\overline B) \bigg\}
\]
Here $\grad$ is the Riemannian gradient and $dV$ is the Riemannian measure.

If $a<A$, then the Pucci operator $\cP_g^+$ is fully nonlinear and has no variational structure.
Applying theorems of Quaas and Sirakov~\cite{QS} and Armstrong~\cite{Arm} shows that the eigenvalue problem \eqref{eveq} admits two principal half-eigenvalues $\lambda^+$ and $\lambda^-$.
See Lemma~\ref{qs} below.
The half-eigenvalue $\lambda^+$ is the unique number such that \eqref{eveq} admits a solution which is positive over $B$.
Similarly the half-eigenvalue $\lambda^-$ is the unique number such that \eqref{eveq} admits a solution which is negative over $B$.
However there is no variational characterization similar to \eqref{rayleigh}.
Quaas and Sirakov~\cite{QS} established the existence and uniqueness of principal half-eigenvalues for a class of fully nonlinear convex elliptic operators which includes the Riemannian Pucci operators considered here.
Armstrong~\cite{Arm} extended this to non-convex operators.
Earlier results of Felmer and Quaas~\cite{FQ}, Quaas~\cite{Q}, and Busca, Esteban, and Quaas~\cite{BEQ} considered principal half-eigenvalues of the Euclidean Pucci operator.

The relationship between the Riemannian metric and the eigenvalues is complicated.
In this article we establish two comparison inequalities.
The first is an extension of Cheng's eigenvalue comparison theorem.
For manifolds with bounded sectional curvature, we prove Cheng's bounds hold for the principal Pucci eigenvalues on geodesic balls of radius less than the injectivity radius.
For manifolds with Ricci curvature bounded below, we prove Cheng's upper bound holds for the principal Pucci eigenvalues on certain small geodesic balls.
To describe this assumption precisely, let $M$ be a Riemannian manifold and let $x_0$ is a point in $M$.
Let $B(x_0,R)$ be a normal geodesic ball about $x_0$ of radius $R>0$.
We say $B(x_0,R)$ is admissible if, for every unit-speed geodesic $\gamma:[0,R] \to M$ such that $\gamma(0)=x_0$ and for every Jacobi field $J$ along $\gamma$ such that $J(0)=0$, the length $|J|$ is monotonic over $[0,R]$.
We refer to Sakai~\cite{S} for background on Riemannian geometry, including definitions of Jacobi fields~\cite[p. 36]{S}, sectional curvature~\cite[p. 43]{S}, Ricci curvature~\cite[p. 45]{S}, and injectivity radius~\cite[p. 110]{S}.

\begin{Theorem}
\label{cheng}
Let $M$ be a Riemannian manifold of dimension $n$.
Let $B_M(x_0,R)$ be a geodesic ball of radius $R>0$ about a point $x_0$ in $M$.
Fix $K$ in $\R$ and let $M_K$ be the complete simply connected space form of dimension $n$ and constant sectional curvature $K$.
Assume $R$ is less than the injectivity radius of $x_0$ in $M$ and less than the injectivity radius of $M_K$. 
Let $B_K(R)$ be a geodesic ball of radius $R$ in $M_K$.
\begin{enumerate}
\item
If the sectional curvature $K_M$ of $M$ satisfies $K_M \le K$, then
\[
\label{cheng1}
	\lambda^+ \Big( B_M(x_0,R) \Big) \ge \lambda^+ \Big( B_K(R) \Big)
\]
and
\[
\label{cheng2}
	\lambda^- \Big( B_M(x_0,R) \Big) \ge \lambda^- \Big( B_K(R) \Big)
\]
If equality holds in \eqref{cheng1} or \eqref{cheng2}, then $B_M(x_0,R)$ is isometric to $B_K(R)$.
\item
If the sectional curvature $K_M$ of $M$ satisfies $K_M \ge K$, then
\[
\label{cheng3}
	\lambda^+ \Big( B_M(x_0,R) \Big) \le \lambda^+ \Big( B_K(R) \Big)
\]
and
\[
\label{cheng4}
	\lambda^- \Big( B_M(x_0,R) \Big) \le \lambda^- \Big( B_K(R) \Big)
\]
If equality holds in \eqref{cheng3} or \eqref{cheng4}, then $B_M(x_0,R)$ is isometric to $B_K(R)$.
\item
Assume $B_M(x_0,R)$ and $B_K(R)$ are admissible.
If the Ricci curvature $\Ric_M$ of $M$ satisfies $\Ric_M \ge (n-1)K$, then
\[
\label{cheng5}
	\lambda^+ \Big( B_M(x_0,R) \Big) \le \lambda^+ \Big( B_K(R) \Big)
\]
and
\[
\label{cheng6}
	\lambda^- \Big( B_M(x_0,R) \Big) \le \lambda^- \Big( B_K(R) \Big)
\]
If equality holds in \eqref{cheng5} or \eqref{cheng6}, then $B_M(x_0,R)$ is isometric to $B_K(R)$.
\end{enumerate}
\end{Theorem}

This theorem is proven in the next section.
Cheng~\cite{Ch1} first proved these inequalities for the principal Dirichlet eigenvalue of the Laplace-Beltrami operator for geodesic balls of radius less than the injectivity radius.
For manifolds with Ricci curvature bounded below, Cheng~\cite{Ch2} extended the bounds to larger geodesic balls.
We note that the argument in~\cite{Ch2} uses the variational characterization \eqref{rayleigh}, but the argument in~\cite{Ch1} is based on Barta's theorem and does not use \eqref{rayleigh}.
For the Pucci eigenvalues, a theorem of Quaas and Sirakov~\cite{QS} yields an analogue of Barta's theorem.
See Lemma~\ref{barta} below.
Therefore we can apply Cheng's argument from~\cite{Ch1} to obtain Theorem~\ref{cheng}.

There are many results generalizing Cheng's theorem.
Bessa and Montenegro~\cite{BM} weakened the geometric assumptions on the geodesic ball.
Freitas, Mao, and Salavessa~\cite{FMS} built different model spaces to use in place of the complete simply connected space form of constant sectional curvature.
Analogues for eigenvalue problems of other operators have been considered, including the Laplacian on differential forms by Dodziuk~\cite{D}, the Steklov operator by Escobar~\cite{E}, a biharmonic Steklov operator by Raulot and Savo~\cite{RS}, the $p$-Laplacian by Mao~\cite{M}, and the Laplacian with drift by Ferreira and Salavessa~\cite{FS}.

A similar argument to the one we use to prove Theorem~\ref{cheng} also yields a second comparison result concerning $O(n)$-invariant hypersurfaces in $\R^{n+1}$ with one smooth boundary component.
We prove that the principal Pucci eigenvalues of such a hypersurface are smaller than the principal Pucci eigenvalues of an $n$-dimensional Euclidean ball with the same boundary.

\begin{Theorem}
\label{af}
Let $\Sigma$ be a smooth connected $O(n)$-invariant hypersurface in $\R^{n+1}$ with one smooth boundary component of radius $R>0$.
Let $B_0(R)$ be an open ball in $\R^n$ of radius $R$, equipped with the Euclidean metric.
Then
\[
\label{af1}
	\lambda^+(\Sigma) \le \lambda^+ \Big( B_0(R) \Big)
\]
and
\[
\label{af2}
	\lambda^-(\Sigma) \le \lambda^- \Big( B_0(R) \Big)
\]
If equality holds in \eqref{af1} or \eqref{af2}, then $\Sigma$ is isometric to $B_0(R)$.
\end{Theorem}

This theorem is proven in the next section.
For the principal Dirichlet eigenvalue of the Laplace-Beltrami operator, these inequalities follow from a theorem of Abreu and Freitas~\cite{AF} if $n=2$ and from a theorem of Colbois, Dryden, and El Soufi~\cite{CDE} if $n \ge 3$.
In fact these results yield stronger inequalities, with a larger ball whose area is equal to the area of $\Sigma$.
Their arguments use the variational characterization \eqref{rayleigh}.
For the principal Pucci eigenvalues, we use an analogue of Barta's theorem to obtain Theorem~\ref{af}.

The results of Abreu and Freitas~\cite{AF} and Colbois, Dryden, and El Soufi~\cite{CDE} also apply to higher order radial eigenvalues of the Laplace-Beltrami operator.
For a radially symmetric metric, the existence of higher order radial half-eigenvalues of the Riemannian Pucci operator follows from a theorem of Esteban, Felmer, and Quaas~\cite{EFQ}.
However the proof of Theorem~\ref{af} relies crucially on the radial monotonicity of the eigenfunctions, so it does not apply to higher order radial half-eigenvalues.
Similarly, the eigenvalues of the Laplacian on an annulus are larger than any other immersed surface of revolution in $\R^3$ with the same boundary~\cite{Ari}.
However the proof of Theorem~\ref{af} does not apply to the principal half-eigenvalues of the Pucci operator on an annulus.

\section{Proofs}
In this section we prove Theorem~\ref{cheng} and Theorem~\ref{af}.
We first express the eigenvalue problem \eqref{eveq} in coordinates and observe that the existence of principal half-eigenvalues $\lambda^+$ and $\lambda^-$ follows from theorems of Quaas and Sirakov~\cite{QS}.
For a positive integer $k$, let $\cS_k$ be the space of $k \times k$ real symmetric matrices.
Define $m^+:\R \to \R$ by
\[
	m^+(x) =
	\begin{cases}
		Ax & x \ge 0 \\
		ax & x \le 0 \\
	\end{cases}
\]
For a matrix $P$ in $\cS_k$, let $\mu_1, \ldots, \mu_k$ be the eigenvalues of $P$ and define
\[
	m^+(P) = \sum_{j=1}^k m^+(\mu_j)
\]
Let $B$ be an open ball in $\R^n$ of finite radius, and let $E_p(B)=C(\overline B) \cap W^{2,p}_{loc}(B)$ for $p < \infty$.
Let $E(B)$ be the set of functions which are in $E_p(B)$ for every $p < \infty$.
By a solution of \eqref{eveq} in $E(B)$, we mean an almost everywhere solution or, equivalently, an $L^n$-viscosity solution.
For background on $L^n$-viscosity solutions, we refer to Caffarelli, Crandall, Kocan, and Swiech~\cite{CCKS}.

\begin{Lemma}
\label{qs}
Let $B$ be an open ball in $\R^n$ of finite radius, and let $g$ be a Riemannian metric on $\overline B$.
The Pucci operator $\cP_g^+$ admits principal half-eigenvalues $\lambda^+$ and $\lambda^-$.
The eigenvalue $\lambda^+$ is the unique number such that \eqref{eveq} admits a solution $\phi_+$ in $E(B)$ which is positive over $B$.
The eigenvalue $\lambda^-$ is the unique number such that \eqref{eveq} admits a solution $\phi_-$ in $E(B)$ which is negative over $B$.
Moreover, the eigenfunctions $\phi_+$ and $\phi_-$ are unique up to scalar multiplication, and the eigenvalues $\lambda^+$ and $\lambda^-$ are positive.
\end{Lemma}

\begin{proof}
The Riemannian Hessian of $\phi$, denoted $\Hess_g \phi$, is a $(0,2)$-tensor on $B$ such that
\[
\label{hessdef}
	\Hess_g \phi(X,Y) = XY \phi - (\grad_X Y) \phi
\]
An eigenvalue and an eigenvector of $\Hess_g \phi$ at a point $x$ in $B$ are a number $\mu$ and a vector $X$ in $T_x B$ such that for all vectors $Y$ in $T_x B$,
\[
\label{hessevdef}
	\Hess_g \phi(X,Y) = \mu \cdot \langle X, Y \rangle_g
\]
Let $X_1, X_2, \ldots, X_n$ be the standard Euclidean basis.
Let $\Gamma_{ij}^k$ be the Christoffel symbols of the metric, defined so that for each $i$ and $j$,
\[
	\grad_{X_i} X_j = \sum_k \Gamma_{ij}^k X_k
\]
Then
\[
	\Hess_g \phi(X_i,X_j) = \del_i \del_j \phi - \sum_{k=1}^n \Gamma_{ij}^k \del_k \phi
\]
Define $D^2 \phi: B \to \cS_n$ by $D^2 \phi = [\del_i \del_j \phi]$.
Define $G: B \to \cS_n$ by $G = [ g_{ij} ]$ and note that $G$ is positive definite.
For a vector $p=(p_1, \ldots, p_n)$ in $\R^n$, define a matrix $\Gamma(p)$ in $\cS_n$ by
\[
	\Gamma(p) = \bigg[ \sum_{k=1}^n \Gamma_{ij}^k p_k \bigg]
\]
Let $D \phi$ be the Euclidean gradient of $\phi$.
Then the eigenvalues of $\Hess_g \phi$ are the eigenvalues of the matrix
\[
	G^{-1/2} \Big( D^2 \phi - \Gamma(D\phi) \Big) G^{-1/2}
\]
Therefore
\[
	\cP_g^+ \phi = m^+ \bigg( G^{-1/2} \Big( D^2 \phi-\Gamma(D \phi) \Big) G^{-1/2} \bigg)
\]
Define $F: \cS_n \times \R^n \times \overline B \to \R$ by
\[
	F(M,p,x) = m^+ \bigg( G^{-1/2} \Big( M-\Gamma(p) \Big) G^{-1/2} \bigg)
\]
Note that $\cP_g^+ \phi = F(D^2 \phi, D \phi, x)$.
Therefore applying theorems of Quaas and Sirakov~\cite[Theorem~1.1, Theorem~1.2, and Remark~1 after Theorem~1.4]{QS} yields the lemma.
\end{proof}

We note that Quaas and Sirakov \cite[p. 108]{QS} define the eigenvalues $\lambda^+$ and $\lambda^-$ using the characterization of Berestycki, Nirenberg, and Varadhan~\cite{BNV}.
That is, $\lambda^+$ is defined to be the supremum of values $\lambda$ such that there exists a continuous function $\psi: \overline B \to \R$ which is positive over $B$ and satisfies
\[
	\cP_g^+ \phi + \lambda \phi \le 0
\]
The eigenvalue $\lambda^-$ is defined similarly.

Next we show that an analogue of Barta's theorem follows from a theorem of Quaas and Sirakov~\cite{QS}.

\begin{Lemma}
\label{barta}
Let $B$ be an open ball in $\R^n$ of finite radius.
Let $g$ be a Riemannian metric on $\overline B$.
Let $\lambda^+$ and $\lambda^-$ be the principal half-eigenvalues of $\cP_g^+$ with corresponding eigenfunctions $\phi_+$ and $\phi_-$ given by Lemma \ref{qs}.
Let $\psi_+$ be a function in $E(B)$ which is positive over $B$ and vanishes on the boundary $\del B$.
Let $\psi_-$ be a function in $E(B)$ which is negative over $B$ and vanishes on the boundary $\del B$.
Then
\[
\label{barta1}
	\einf_B \frac{-\cP_g^+ \psi_+}{\psi_+} \le \lambda^+ \le \esup_B \frac{-\cP_g^+ \psi_+}{\psi_+}
\]
and
\[
\label{barta2}
	\einf_B \frac{-\cP_g^+ \psi_-}{\psi_-} \le \lambda^- \le \esup_B \frac{-\cP_g^+ \psi_-}{\psi_-}
\]
Furthermore, if equality holds in either inequality in \eqref{barta1}, then $\psi_+$ is a scalar multiple of $\phi_+$.
Similarly, if equality holds in either inequality in \eqref{barta2}, then $\psi_-$ is a scalar multiple of $\phi_-$.
\end{Lemma}

\begin{proof}
Define
\[
	\mu^+ = \einf_B \frac{-\cP_g^+ \psi_+}{\psi_+}
\]
For almost every $x$ in $B$,
\[
	\cP_g^+ \psi_+(x) + \mu^+ \psi_+(x) \le 0
\]
Applying a theorem of Quaas and Sirakov~\cite[Theorem 4.1]{QS} shows that $\lambda^+ \ge \mu^+$.
Furthermore if equality holds then $\psi_+$ is a scalar multiple of $\phi_+$.
This proves the first inequality in \eqref{barta1}.
Similar arguments yield the second inequality in \eqref{barta1} and the two inequalities in \eqref{barta2}.
\end{proof}

In particular, we note that the eigenvalues $\lambda^+$ and $\lambda^-$ can be expressed by max-min and min-max formulas.
Namely,
\[
	\lambda^+ = \sup \: \einf_B \frac{-\cP_g^+ \psi}{\psi}
\]
Here the supremum is taken over functions $\psi$ in $E(B)$ which are positive over $B$ and vanish on the boundary $\del B$.
Similarly,
\[
	\lambda^+ = \inf \: \esup_B \frac{-\cP_g^+ \psi}{\psi}
\]
Here the infimum is taken over functions $\psi$ in $E(B)$ which are positive over $B$ and vanish on the boundary $\del B$.
Similar formulas hold for $\lambda^-$.

Let $R>0$ and let $B$ be an open ball about the origin in $\R^n$ of radius $R$.
In order to prove Theorem~\ref{cheng} and Theorem~\ref{af}, it suffices to only consider metrics on $\overline B$ which arise when $B$ is identified with a normal geodesic ball of radius $R$ in some Riemannian manifold using geodesic normal coordinates.
For a defintion of geodesic normal coordinates, see Sakai~\cite[p. 33]{S}.
If $g$ is such a metric on $\overline B$, then we refer to $g$ as a normal metric on $\overline B$.
In the next lemma, we give a formula for $\cP_g^+ f$ if $g$ is a normal metric on $\overline B$ and $f$ is a radial function.
To state this formula, fix a point $x \neq 0$ in $B$.
Let $r_0=|x|$, and let $\gamma$ be the unit speed geodesic with $\gamma(0)=0$ and $\gamma(r_0)=x$.
Let $V$ be the set of vectors in $T_x B$ which are orthogonal to $\gamma$.
Define a symmetric bilinear form $\dot g_x$ on $V$ as follows.
Fix two vectors $w$ and $z$ in $V$.
Let $W$ and $Z$ be the Jacobi fields along $\gamma$ which are orthogonal to $\gamma$ such that $W(0)=Z(0)=0$ and $W(r_0)=w$ and $Z(r_0)=z$.
Then define
\[
	\dot g_x (w,z) = \frac{1}{2} \frac{d}{dt} \langle W, Z \rangle_g \Big|_{t=r_0}
\]
An eigenvalue and eigenvector of $\dot g_x$ are a number $\zeta$ and a vector $w$ in $V$ such that for all vectors $z$ in $V$,
\[
	\dot g_x(w,z) = \zeta \langle w, z \rangle_g
\]
If $\zeta_1, \ldots, \zeta_{n-1}$ are the eigenvalues of $\dot g_x$, define
\[
	m^+(\dot g_x) = \sum_{j=1}^{n-1} m^+(\zeta_j)
\]
For a function $f:B \to \R$, we use the notation $f'$ and $f''$ for the first and second radial derivatives of $f$.
Note that a radial function in $E(B)$ is twice differentiable at almost every point in $B$.

\begin{Lemma}
\label{normal}
Let $B$ be an open ball about the origin in $\R^n$ of finite radius.
Let $g$ be a normal metric on $\overline B$, and let $f$ be a radial function in $E(B)$.
For almost every $x$ in $B$,
\[
\label{radial1}
	\cP_g^+ f(x) = m^+ \Big( f''(x) \Big) + m^+ \Big( f'(x) \dot g_x \Big)
\]
This holds at any point $x \neq 0$ in $B$ such that $f$ is twice differentiable at $x$.
\end{Lemma}

\begin{proof}
Fix a point $x \neq 0$ in $B$ and assume $f$ is twice differentiable at $x$.
Let $\zeta_1, \ldots, \zeta_{n-1}$ be the eigenvalues of $\dot g_x$.
Let $u$ and $v$ be vectors in $T_x B$.
Let $r_0=|x|$, and let $\gamma$ be the unit speed geodesic with $\gamma(0)=0$ and $\gamma(r_0)=x$.
Write $u=b \gamma'(r_0)+w$ and $v=c \gamma'(r_0)+z$ where $b, c$ are in $\R$ and $w, z$ are orthogonal to $\gamma$.
Expressing the Hessian in geodesic polar coordinates shows that
\[
	\Big( \Hess_g f(x) \Big)(u,v) = f''(x) bc + f'(x) \dot g_x(w,z)
\]
By orthogonality,
\[
	\langle u,v \rangle_g = bc + \langle w,z \rangle_g
\]
Therefore the eigenvalues of $\Hess_g f(x)$ are $f''(x)$ and the eigenvalues of $f'(x) \dot g_x$.
This proves \eqref{radial1}.
\end{proof}

The following lemma shows that bounds on the sectional curvature of a normal metric $g$ yield bounds on $m^+(\dot g_x)$.

\begin{Lemma}
\label{rauch}
Let $R>0$ and let $B_R$ be an open ball about the origin in $\R^n$ of radius $R$.
Fix $K$ in $\R$ and let $M_K$ be the complete simply connected space form of dimension $n$ and constant sectional curvature $K$.
If $K>0$, then assume that $R$ is less than the injectivity radius of $M_K$.
Let $h$ be the normal metric on $\overline{B_R}$ such that $(B_R,h)$ is isometric to a geodesic ball of radius $R$ in $M_K$.
Let $g$ be a normal metric on $\overline{B_R}$.
\begin{enumerate}
\item
Assume the sectional curvature $K_g$ with respect to $g$ satisfies $K_g \le K$.
If $x \neq 0$ is in $B_R$, then $m^+(\dot g_x) \ge m^+(\dot h_x)$ and $m^+(-\dot g_x) \le m^+(-\dot h_x)$.
Moreover, if $m^+(\dot g_x) = m^+(\dot h_x)$ for all $x \neq 0$ in $B_R$, then $g$ is equal to $h$ over $\overline B_R$.
Similarly, if $m^+(-\dot g_x) = m^+(-\dot h_x)$ for all $x \neq 0$ in $B_R$, then $g$ is equal to $h$ over $\overline B_R$.
\item
Assume the sectional curvature $K_g$ with respect to $g$ satisfies $K_g \ge K$.
If $x \neq 0$ is in $B_R$, then $m^+(\dot g_x) \le m^+(\dot h_x)$ and $m^+(-\dot g_x) \ge m^+(-\dot h_x)$.
Moreover, if $m^+(\dot g_x) = m^+(\dot h_x)$ for all $x \neq 0$ in $B_R$, then $g$ is equal to $h$ over $\overline B_R$.
Similarly, if $m^+(-\dot g_x) = m^+(-\dot h_x)$ for all $x \neq 0$ in $B_R$, then $g$ is equal to $h$ over $\overline B_R$.

\end{enumerate}
\end{Lemma}

\begin{proof}
Let $\zeta_1, \ldots, \zeta_{n-1}$ be the eigenvalues of $\dot g_x$.
Fix $i=1,2,\ldots,n-1$ and let $w$ be an eigenvector of $\dot g_x$ corresponding to $\zeta_i$.
Let $r_0=|x|$.
Let $\gamma$ be a unit speed geodesic with respect to $g$ such that $\gamma(0)=0$ and $\gamma(r_0)=x$.
Let $W$ be the Jacobi field along $\gamma$ with respect to $g$ such that $W(0)=0$ and $W(r_0)=w$.
We have
\[
\label{rauch1}
	\zeta_i = \frac{\dot g_x (w,w)}{|w|_g^2} = \frac{\frac{d}{dt} \Big( | W |_g^2 \Big)_{t=r_0}}{2 |w|_g^2}
\]
The $n-1$ eigenvalues of $\dot h_x$ are all equal.
Let $\zeta_h$ denote this value.
The curve $\gamma$ is also a unit speed geodesic with respect to $h$, because $g$ and $h$ are both normal metrics on $B_R$.
Moreover $W$ is also a Jacobi field along $\gamma$ with respect to $h$, and $w$ is orthogonal to $\gamma$ with respect to $h$.
Furthermore the covariant derivative $W'(0)$ is the same for $g$ and $h$.
We have
\[
\label{rauch2}
	\zeta_h = \frac{\dot h_x (w,w)}{|w|_h^2} = \frac{\frac{d}{dt} \Big( | W |_h^2 \Big)_{t=r_0}}{2 |w|_h^2}
\]
Define a function $f:(0,r_0] \to \R$ by
\[
\label{rauchf}
	f(t) = \frac{|W(t)|_g^2}{|W(t)|_h^2}
\]
If $K_g \le K$, then $f$ is non-decreasing over $(0,r_0]$ by the Rauch comparison theorem.
For a statement of the Rauch comparison theorem, see Sakai~\cite[p. 149, Theorem 2.3]{S}.
In particular $f'$ is non-negative over $(0, r_0]$.
At $r_0$, this yields
\[
	|w|_h^2 \frac{d}{dt} \Big( | W |_g^2 \Big)_{t=r_0} \ge |w|_g^2 \frac{d}{dt} \Big( | W |_h^2 \Big)_{t=r_0}
\]
By \eqref{rauch1} and \eqref{rauch2}, this shows that $\zeta_i \ge \zeta_h$.
The monotonicity of $m^+$ implies that $m^+(\zeta_i) \ge m^+(\zeta_h)$ and $m^+(-\zeta_i) \le m^+(-\zeta_h)$.
This holds for each $i=1,2,\ldots,n-1$, so $m^+(\dot g_x) \ge m^+(\dot h_x)$ and $m^+(-\dot g_x) \le m^+(-\dot h_x)$.
Moreover if $m^+(\dot g_x) = m^+(\dot h_x)$ or $m^+(-\dot g_x) = m^+(-\dot h_x)$, then $\zeta_i = \zeta_h$ for all $i=1,2,\ldots,n-1$.
If this holds at every $x \neq 0$ in $B_R$, then the function $f$ defined in \eqref{rauchf} is constant.
Therefore $f(t)=1$ for all $t$ in $(0,r_0]$, because
\[
	\lim_{t \to 0^+} f(t) = 1
\]
In particular $|w|_g=|w|_h$.
It follows that $g$ is equal to $h$ over $\overline{B_R}$.
This establishes the lemma for the case $K_g \le K$.
A similar argument applies to the case $K_g \ge K$.
\end{proof}

The following lemma shows that lower bounds on the Ricci curvature of a normal metric $g$ on a ball $\overline B$ yield upper bounds on $m^+(\dot g_x)$ for admissible geodesic balls.

\begin{Lemma}
\label{bg}
Let $R>0$ and let $B_R$ be an open ball about the origin in $\R^n$ of radius $R$.
Fix $K$ in $\R$ and let $M_K$ be the complete simply connected space form of dimension $n$ and constant sectional curvature $K$.
If $K>0$, then assume that $R$ is less than the injectivity radius of $M_K$.
Let $h$ be the normal metric on $\overline{B_R}$ such that $(B_R,h)$ is isometric to a geodesic ball of radius $R$ in $M_K$.
Let $g$ be a normal metric on $\overline{B_R}$.
Assume the Ricci curvature $\Ric_g$ with respect to $g$ satisfies $\Ric_g \ge K(n-1)$.
Assume $B_R$ is admissible with respect to $g$ and $h$, as defined in the paragraph preceding Theorem~\ref{cheng}.
If $x \neq 0$ is a point in $B_R$, then $m^+(\dot g_x) \le m^+(\dot h_x)$ and $m^+(-\dot g_x) \ge m^+(-\dot h_x)$.
Moreover, if $m^+(\dot g_x) = m^+(\dot h_x)$ for all $x \neq 0$ in $B_R$, then $g$ is equal to $h$ over $\overline{B_R}$.
Similarly, if $m^+(-\dot g_x) = m^+(-\dot h_x)$ for all $x \neq 0$ in $B_R$, then $g$ is equal to $h$ over $\overline{B_R}$.
\end{Lemma}

\begin{proof}
Let $\zeta_1, \ldots, \zeta_{n-1}$ be the eigenvalues of $\dot g_x$.
Define
\[
	\tr (\dot g_x) = \sum_{j=1}^{n-1} \zeta_j
\]
The $n-1$ eigenvalues of $\dot h_x$ are all equal.
Let $\zeta_h$ denote this value.
Define $\tr( \dot h_x ) = (n-1) \zeta_h$.
By \eqref{rauch1} and \eqref{rauch2}, the hypothesis that $B_R$ is admissible with respect to $g$ and $h$ implies that the eigenvalues of $\dot g_x$ and $\dot h_x$ are non-negative.
Therefore $m^+(\dot g_x) = A \tr (\dot g_x)$ and $m^+(\dot h_x) = A \tr  (\dot h_x)$.
Also $m^+(-\dot g_x) = -a \tr (\dot g_x)$ and $m^+(-\dot h_x) = -a \tr (\dot h_x)$.
Therefore it suffices to show that $\tr (\dot g_x) \le \tr (\dot h_x)$.
Define $|g|=\det [g_{ij}]$ and $|h|=\det [h_{ij}]$, and note that $|g|'= |g| \tr( \dot g_x)$ and $|h|' = |h| \tr( \dot h_x)$.
By the Bishop-Gromov theorem,
\[
	\bigg( \frac{|g|}{|h|} \bigg)' \le 0
\]
For a statement of the Bishop-Gromov theorem, see Sakai~\cite[pp. 154-155, Theorem 3.1]{S}.
Therefore $\tr (\dot g_x) \le \tr (\dot h_x)$.
Moreover, if $\tr(\dot g_x) = \tr(\dot h_x)$ for every $x \neq 0$ in $B_R$, then $|g|=|h|$ over $B_R$.
In particular the volume of $B_R$ with respect to $g$ is equal to the volume with respect to $h$, so $g$ is equal to $h$ over $\overline{B_R}$ by the Bishop-Gromov theorem.
For this aspect of the Bishop-Gromov theorem, see Sakai~\cite[p. 155, Corollary 3.2]{S}.
\end{proof}

For a radially symmetric normal metric $g$, we observe that the eigenfunctions of $\cP_g^+$ given by Lemma~\ref{qs} are radially symmetric and radially monotonic.

\begin{Lemma}
\label{efrad}
Let $B$ be an open ball about the origin in $\R^n$ of finite radius.
Let $g$ be a radially symmetric normal metric on $\overline B$.
The eigenfunctions $\phi_+$ and $\phi_-$ of $\cP_g^+$ given by Lemma~\ref{qs} are radial, continuously differentiable over $B$, and twice differentiable at almost every point in $B$.
Moreover $\phi_+'(x) \le 0$ and $\phi_-'(x) \ge 0$ for every $x$ in $B$.
Furthermore $\phi_+'(x) \neq 0$ and $\phi_-'(x) \neq 0$ for almost every $x$ in $B$.
\end{Lemma}

\begin{proof}
The eigenfunction $\phi_+$ is in $E(B)$ by Lemma~\ref{qs}.
In particular $\phi_+$ is in $W^{2,p}_{loc}(B)$ for all $p<\infty$.
Therefore $\phi_+$ is continuously differentiable over $B$ by Sobolev embedding.
Additionally $\phi_+$ is radial, because it is unique up to scalar multiplication.
This implies that $\phi_+$ is twice differentiable at almost every point in $B$.
Since $\phi_+$ is an $L^n$-viscosity solution of the eigenvalue equation \eqref{eveq}, it follows that $\phi_+$ has no interior local minima.
Therefore $\phi_+'(x) \le 0$ for every $x$ in $B$.
Define a set
\[
	Z_+ = \Big\{ x \in B : \phi_+'(x) = 0 \Big\}
\]
Suppose $Z_+$ has positive measure.
The isolated points of $Z_+$ are countable, hence have measure zero.
Therefore there is a point $x_0$ in $Z_+$ which satisfies the following four properties.
First, $x_0$ is a limit point of $Z_+$.
Second, $x_0$ is not the origin.
Third, $\phi_+$ is twice differentiable at $x_0$.
Fourth,
\[
\label{efrad1}
	\cP_g^+ \phi_+(x_0) = -\lambda^+ \phi_+(x_0)
\]
Note that $\phi_+'(x_0)=0$, because $x_0$ is in $Z_+$.
Since $\phi_+$ is twice differentiable at $x_0$, it follows that $\phi_+''(x_0)=0$, because $x_0$ is a limit point of $Z_+$.
Therefore $\cP_g^+ \phi_+(x_0)=0$ by Lemma~\ref{normal}.
Additionally $\lambda^+$ is positive by Lemma~\ref{qs}.
Hence $\phi_+(x_0)=0$, by \eqref{efrad1}.
However $\phi_+(x)>0$ for all $x$ in $B$, by Lemma~\ref{qs}.
This contradiction proves that $Z_+$ has measure zero.
That is $\phi_+'(x) \neq 0$ for almost every $x$ in $B$.
This completes the proof for $\phi_+$, and a similar argument applies to $\phi_-$.
\end{proof}

Now we can prove Theorem~\ref{cheng}.

\begin{proof}[Proof of Theorem 1.1]
We first prove \eqref{cheng1}.
Let $B_R$ be an open ball about the origin in $\R^n$ of radius $R$.
Use geodesic normal coordinates to identify $B_M(x_0,R)$ with $B_R$, and let $g$ be the induced normal metric on $\overline{B_R}$.
Similarly, use geodesic normal coordinates to identify $B_K(R)$ with $B_R$, and let $h$ be the induced normal metric on $\overline{B_R}$.
Let $\phi_+$ be the positive eigenfunction of $\cP_h^+$ given by Lemma~\ref{qs}.
By Lemma~\ref{efrad}, the eigenfunction $\phi_+$ is radial, continuously differentiable over $B_R$, and twice differentiable at almost every point in $B_R$.
Moreover $\phi_+'(x) \le 0$ for every $x$ in $B_R$.
By Lemma~\ref{normal}, for almost every $x$ in $B_R$,
\[
\label{chengmg+}
	\cP_g^+ \phi_+(x) = m^+ \Big( \phi_+''(x) \Big) + m^+ \Big( \phi_+'(x) \dot g_x \Big)
\]
and
\[
\label{chengmh+}
	\cP_h^+ \phi_+(x) = m^+ \Big( \phi_+''(x) \Big) + m^+ \Big( \phi_+'(x) \dot h_x \Big)
\]
Note that $m^+(-\dot g_x) \le m^+(-\dot h_x)$ for every $x \neq 0$ in $B_R$ by Lemma~\ref{rauch}.
Hence for every $x \neq 0$ in $B_R$,
\[
	m^+ \Big( \phi_+'(x) \dot g_x \Big) \le m^+ \Big( \phi_+'(x) \dot h_x \Big)
\]
By \eqref{chengmg+} and \eqref{chengmh+}, this shows that for almost every $x$ in $B_R$,
\[
	\cP_g^+ \phi_+(x) \le \cP_h^+ \phi_+(x)
\]
By Lemma~\ref{barta},
\[
\label{lglh}
	\lambda^+(g) \ge \einf_B \frac{-\cP_g^+ \phi_+}{\phi_+} \ge \einf_B \frac{-\cP_h^+ \phi_+}{\phi_+} = \lambda^+(h)
\]
This completes the proof of \eqref{cheng1}.

Next we prove that if $\lambda^+(g) = \lambda^+(h)$, then $B_M(x_0,R)$ is isometric to $B_K(R)$.
By \eqref{lglh}, the assumption that $\lambda^+(g) = \lambda^+(h)$ implies that
\[
	\lambda^+(g) = \einf_B \frac{-\cP_g^+ \phi_+}{\phi_+}
\]
Therefore $\phi_+$ is also an eigenfunction of $\cP_g^+$, by Lemma~\ref{barta}.
In particular $\cP_g^+ \phi_+(x) = \cP_h^+ \phi_+(x)$ for almost every $x$ in $B$.
Moreover $\phi_+'(x) \neq 0$ for almost every $x$ in $B_R$, by Lemma~\ref{efrad}.
By \eqref{chengmg+} and \eqref{chengmh+}, this shows that $m^+(-\dot g_x) = m^+(-\dot h_x)$ for every $x \neq 0$ in $B$.
Therefore $g$ is equal to $h$ over $\overline{B_R}$ by Lemma~\ref{rauch}.
Hence $B_M(x_0,R)$ is isometric to $B_K(R)$.

Repeating the argument establishes the other statements as well.
For \eqref{cheng5} and \eqref{cheng6}, we use Lemma~\ref{bg} instead of Lemma~\ref{rauch}.
\end{proof}

In the proof of Theorem~\ref{af}, we use the following variant of Lemma~\ref{normal}.

\begin{Lemma}
\label{puccirad}
Let $B$ be an open ball about the origin in $\R^n$ of finite radius.
Let $g$ be a normal metric on $\overline B$ which is $O(n)$-invariant.
Let $\rho: B \to \R$ be the continuous function such that for each $x \neq 0$ in $B$, the sphere about the origin containing $x$ is isometric to a round sphere of radius $\rho(x)$.
Let $f$ be a radial function in $E(B)$.
Then for almost every $x$ in $B$,
\[
\label{puccirad1}
	\cP_g^+ f(x) = m^+ \Big( f''(x) \Big) + (n-1) m^+\bigg( \frac{f'(x) \rho'(x)}{\rho(x)} \bigg)
\]
\end{Lemma}

\begin{proof}
Let $x \neq 0$ be a point in $B$ such that $f$ is twice differentiable at $x$.
Expressing the Hessian in geodesic polar coordinates shows that the eigenvalues $\mu_1, \mu_2, \ldots, \mu_n$ of $\Hess_g f(x)$ are
\[
	\mu_1 = f''(x)
\]
and
\[
	\mu_2 = \ldots = \mu_n = \frac{\rho'(x)}{\rho(x)} f'(x)
\]
Therefore \eqref{puccirad1} holds.
\end{proof}

We conclude the article by proving Theorem~\ref{af}.

\begin{proof}[Proof of Theorem 1.2]
First we prove \eqref{af1}.
Note that there is exactly one point $p$ in $\Sigma$ which is on the axis of symmetry in $\R^{n+1}$.
Let $L$ be the intrinsic distance in $\Sigma$ from $p$ to the boundary $\del \Sigma$.
Let $B_L$ be an open ball about the origin in $\R^n$ of radius $L$.
Use geodesic normal coordinates at $p$ to identify $\Sigma$ with $B_L$ and let $g$ be the induced normal metric on $\overline{B_L}$.
Let $\rho: \overline{B_L} \to \R$ be the continuous function such that for each $x \neq 0$ in $B_L$, the sphere about the origin containing $x$ is isometric to a round sphere of radius $\rho(x)$.
Note that $\rho$ is radially symmetric, smooth away from the origin, and satisfies $|\rho'(x)| \le 1$ for all $x$ in $\overline{B_L}$.
Let $B_R$ be an open ball about the origin in $\R^n$ of radius $R$.
Let $h$ be the Euclidean metric on $B_R$.
Let $\phi_+: \overline{B_R} \to \R$ be the positive eigenfunction of $\cP_h^+$ in $E(B_R)$ given by Lemma~\ref{qs}.
By Lemma~\ref{efrad}, the eigenfunction $\phi_+$ is radial, continuously differentiable over $B_R$, and twice differentiable at almost every point in $B_R$.
Note that $L \ge R$, and let $d=L-R$.
Define a radial function $\psi_+: \overline{B_L} \to \R$ by
\[
	\psi_+(x) =
	\begin{cases}
		\phi_+(|x|-d) & |x| \ge d \\
		\phi_+(0) & |x| \le d \\
	\end{cases}
\]
H\"older's inequality shows that $\psi_+$ is in $E(B_L)$.
Moreover $\psi_+$ is radial, continuously differentiable over $B_L$, and twice differentiable at almost every point in $B_L$.
Furthermore $\psi_+'(x) \le 0$ for every $x$ in $B_L$.
By Lemma~\ref{puccirad}, for almost every $x$ in $B_L$,
\[
\label{afmg+}
	\cP_g^+ \psi_+(x) = m^+ \Big( \psi_+''(x) \Big) - (n-1) \psi_+'(x) m^+ \bigg( -\frac{\rho'(x)}{\rho(x)} \bigg )
\]
Also by Lemma~\ref{puccirad}, for almost every $x$ in $B_R$,
\[
\label{afmh+}
	\cP_h^+ \phi_+(x) = m^+ \Big( \phi_+''(x) \Big) - (n-1) \phi_+'(x) m^+ \bigg( -\frac{1}{|x|} \bigg)
\]
Note that $|\rho'(x)| \le 1$ for all $x$ in $B_L$ and $\rho(L)=R$.
In particular $\rho(x) \ge |x|-d$ for all $x$ in $B_L$.
Therefore, if $x$ is a point in $B_L$ such that $|x| > d$, then
\[
	-\frac{\rho'(x)}{\rho(x)} \ge -\frac{1}{|x|-d}
\]
By the monotonicity of $m^+$, if $x$ is a point in $B_L$ such that $|x| > d$, then
\[
	m^+ \bigg( -\frac{\rho'(x)}{\rho(x)} \bigg) \ge m^+ \bigg( -\frac{1}{|x|-d} \bigg)
\]
By \eqref{afmg+} and \eqref{afmh+}, this shows that if $x$ is a point in $B_L$ such that $|x| > d$, then
\[
	\cP_g^+ \psi_+(x) \ge \cP_h^+ \phi_+(|x|-d) 
\]
If $x$ is a point in $B_L$ such that $|x| < d$, then $\psi_+$ is constant on a neighborhood of $x$, so $\cP_g^+ \psi_+(x) = 0$.
Therefore, by Lemma~\ref{barta},
\[
\label{aflhlg}
	\lambda^+(\Sigma) \le \esup_{B_L} \frac{-\cP_g^+ \psi_+(x)}{\psi_+(x)} \le \esup_{B_R} \frac{-\cP_h^+ \phi_+(x)}{\phi_+(x)} = \lambda^+(B_0(R))
\]
This proves \eqref{af1}.

Next we prove that if $\lambda^+(\Sigma)=\lambda^+(B_0(R))$, then $\Sigma$ is isometric to $B_0(R)$.
By \eqref{aflhlg}, the assumption that $\lambda^+(\Sigma)=\lambda^+(B_0(R))$ implies that
\[
	\lambda^+(\Sigma) = \esup_{B_L} \frac{-\cP_g^+ \psi_+(x)}{\psi_+(x)}
\]
Therefore $\psi_+$ is an eigenfunction of $\cP_g^+$ by Lemma~\ref{barta}.
Then $\psi_+'(x) \neq 0$ for almost every $x$ in $B_L$, by Lemma~\ref{efrad}.
However $\psi_+'(x)=0$ for every $x$ in $B_L$ satisfying $|x|<d$.
Therefore $d=0$, i.e. $L=R$.
This implies that $\Sigma$ is isometric to $B_0(R)$.
A similar argument yields \eqref{af2}.
\end{proof}

\end{document}